\documentclass[12pt]{article}
\usepackage[
  top=2.5cm,
  bottom=3cm,
  left=2.5cm,
  right=2.5cm,
  centering
]{geometry}
\usepackage[utf8]{inputenc}
\usepackage[T1]{fontenc}
\usepackage{authblk}
\usepackage{lmodern}
\usepackage{url}
\usepackage{amsmath,amssymb,amsthm,mathtools}
\usepackage{enumitem}
\usepackage{tikz}
\usetikzlibrary{calc,positioning}
\setlist{nosep}

\usepackage[
colorlinks=true,
linkcolor=blue,
citecolor=blue,
urlcolor=cyan,
bookmarks=true,
pdftitle={Your Document Title},
pdfauthor={Your Name}
]{hyperref}

\newtheorem{theorem}{Theorem}
\newtheorem{lemma}[theorem]{Lemma}
\newtheorem{proposition}[theorem]{Proposition}
\newtheorem{definition}[theorem]{Definition}

\newcommand{\Ree}{\operatorname{Re}}
\newcommand{\Imm}{\operatorname{Im}}

\title{Stability of independence polynomials of spiders}

\author[1]{Lei Zhang}
\author[2,*]{Jianhua Tu}
\affil[1]{\small Department of Mathematics, Taiyuan University of Technology,
  Taiyuan 030024, China}
\affil[2]{\small School of Mathematics and Statistics, Beijing Technology and Business University, Beijing 100048, China}
\date{}

\begin{document}
\maketitle

\begingroup
\renewcommand{\thefootnote}{\fnsymbol{footnote}}
\footnotetext[1]{Corresponding author.\\
\indent\ \ E-mail addresses: zhanglei04@tyut.edu.cn (L. Zhang),
tujh81@163.com (J. Tu).}
\endgroup

\begin{abstract}
	For a graph $G$, let $i_k(G)$ denote the number of independent sets
	of cardinality $k$, and let
	\[
	I(G,z)=\sum_{k\ge0} i_k(G)z^k
	\]
  be its independence polynomial. Following Brown and Cameron \cite{BrownCameron2018}, a graph is called stable if all zeros
  of its independence polynomial lie in the closed left half-plane.
  They proved that every star is stable, but also constructed nonstable
  trees. They then asked for a characterization of stable trees.
	In this paper, we extend and strengthen their result by proving that every spider,
	obtained from a star by arbitrary and possibly nonuniform subdivisions
	of its edges, has all its independence roots in the open left
	half-plane. Hence, every spider is stable.
\end{abstract}
\section{Introduction}

All graphs considered in this paper are finite and simple. For a graph
$G$, let $i_k(G)$ denote the number of independent sets of
cardinality $k$. The \emph{independence polynomial} of $G$, introduced
by Gutman and Harary~\cite{GutmanHarary1983}, is
\[
I(G,z)=\sum_{k=0}^{\alpha(G)} i_k(G)z^k,
\]
where $\alpha(G)$ denotes the independence number of $G$. The zeros of
$I(G,z)$ are called the \emph{independence roots} of $G$.

The location of independence roots has been studied extensively. One
of the fundamental results in this direction is the following theorem
of Chudnovsky and Seymour \cite{ChudnovskySeymour2007}.

\begin{theorem}\cite{ChudnovskySeymour2007}
	Every independence root of a claw-free graph is real.
\end{theorem}

Since the coefficients of an independence polynomial are positive,
every real independence root is negative. Thus the theorem gives a
large class of graphs whose independence roots lie entirely on the
negative real axis. For general graphs, however, the picture is very
different. Brown, Hickman and Nowakowski
~\cite{BrownHickmanNowakowski2004} showed that, taken over all graphs,
independence roots are dense in the complex plane, while Brown and
Nowakowski~\cite{BrownNowakowski2005} proved that almost all graphs
have a nonreal independence root. See also
\cite{BrownDilcherNowakowski2000,LevitMandrescuSurvey2005} for further
results and background on independence roots.

A natural weakening of real-rootedness is to require all roots to lie
in the closed left half-plane. Brown and
Cameron~\cite{BrownCameron2018} initiated a systematic study of this
question. Following their terminology, a polynomial is called
\emph{stable} if all of its zeros lie in the closed left half-plane,
and a graph is called stable if its independence polynomial is stable.

\begin{theorem}\cite{BrownCameron2018}
	The following statements hold.
	\begin{enumerate}[label=\textup{(\roman*)}]
		\item Every graph $G$ with $\alpha(G)\le3$ is stable.
		\item Every star $K_{1,n}$ is stable.
		\item For every $R>0$, there exists a tree $T$ whose independence
		polynomial has a root $\zeta$ satisfying $\Ree\zeta>R$.
	\end{enumerate}
\end{theorem}

Brown and Cameron also asked whether every complete bipartite graph is
stable. This question was recently answered by Chen, Ning and
Tu~\cite{ChenNingTu2025}. They proved that $K_{2,n}$ and $K_{3,n}$
are stable and that, for every fixed integer $k\ge0$, the graphs
$K_{m,m+k}$ are stable for all sufficiently large $m$. On the other
hand, for every fixed rational $\ell>1$, the graphs $K_{m,\ell m}$
are nonstable for all sufficiently large $m$ for which $\ell m$ is
an integer.

The present paper returns to the question of stability for trees.
Brown and Cameron reported that computations found all trees of order at
most $20$ to be stable. Together with parts (ii) and (iii) above, this illustrates a
sharp contrast: stars form a stable family, whereas trees in general
can have independence roots arbitrarily far into the open right
half-plane. It is therefore natural to ask what happens when a star
is modified while retaining a simple tree structure. The most basic
such modification is to subdivide its edges, possibly by different
amounts. The resulting trees are precisely the spiders.

\begin{definition}\label{def:spider}
	Let $d\ge1$ and let $\ell_1,\ldots,\ell_d$ be positive integers.
	The spider $S(\ell_1,\ldots,\ell_d)$ is the tree with a distinguished
	central vertex $c$ such that deleting $c$ leaves $d$ pairwise disjoint
	paths $P_{\ell_1},\ldots,P_{\ell_d}$, where $P_n$ denotes the path
	on $n$ vertices. Thus the $j$th leg has $\ell_j$ edges from $c$ to
	its leaf.
\end{definition}

Spiders have also appeared in the study of coefficient properties of
independence polynomials; see, for example, Levit and
Mandrescu~\cite{LevitMandrescu2003}. Here we consider the location of
their independence roots. Our main result shows that arbitrary, nonuniform subdivisions of the
edges of a star preserve stability.

\begin{theorem}\label{thm:main}
	For every $d\ge1$ and all positive integers
	$\ell_1,\ldots,\ell_d$,
	\[
	I\bigl(S(\ell_1,\ldots,\ell_d),z\bigr)\neq0
	\qquad\text{whenever }\Ree z\ge0.
	\]
	Hence every independence root of a spider has strictly negative real
	part, and thus every spider is stable.
\end{theorem}

%Thus the stability of stars persists under arbitrary and nonuniform
%subdivisions of their edges. Moreover, the conclusion is stronger
%than stability in the sense of Brown and Cameron, since independence
%roots on the imaginary axis are excluded as well.

The rest of this paper is organized as follows. Section~\ref{sec:prelim} develops the necessary facts about path
independence polynomials and the associated ratios.
Section~\ref{sec:mainproof} proves Theorem~\ref{thm:main}, and
Section~\ref{sec:phasegain} establishes the technical estimate used in
that proof.

\section{Path Polynomials and Preliminary Estimates}
\label{sec:prelim}

We begin with the independence polynomials of paths. For $n\ge0$, let
\[
F_n(z):=I(P_n,z).
\]
By considering whether an independent set contains an end
vertex of the path, we obtain
\begin{equation}\label{eq:path-rec}
	F_0(z)=1,\qquad F_1(z)=1+z,\qquad
	F_n(z)=F_{n-1}(z)+zF_{n-2}(z)\quad(n\ge2).
\end{equation}

The coefficients of $F_n$ also have a simple form. An independent
$k$-set of $P_n$ can be represented by integers
\[
1\le v_1<\cdots<v_k\le n,
\qquad
v_{r+1}\ge v_r+2.
\]
The substitution $u_r=v_r-(r-1)$ gives a bijection between such
sequences and the $k$-subsets of $\{1,\ldots,n+1-k\}$. Hence
\begin{equation}\label{eq:path-coeff}
	F_n(z)=
	\sum_{k=0}^{\lfloor(n+1)/2\rfloor}
	\binom{n+1-k}{k}z^k.
\end{equation}
In particular, for $m\ge1$,
\begin{equation}\label{eq:path-degree}
	\begin{aligned}
		\deg F_{2m-1}&=m,
		& [z^m]F_{2m-1}&=1,\\
		\deg F_{2m}&=m,
		& [z^m]F_{2m}&=m+1.
	\end{aligned}
\end{equation}

The zeros of the independence polynomial of a path are known
explicitly. We use the following formula from Alikhani and
Peng~\cite{AlikhaniPeng2011}. The real-rootedness of these polynomials also follows from the earlier work
of Brown, Hickman and Nowakowski~\cite{BrownHickmanNowakowski2004}.

\begin{proposition}\cite{AlikhaniPeng2011, BrownHickmanNowakowski2004}\label{prop:path-roots}
	For every $n\ge1$, the zeros of $F_n$ are the simple negative
	numbers
	\begin{equation}\label{eq:path-roots}
		-\frac{1}
		{4\cos^2\!\left(\frac{j\pi}{n+2}\right)},
		\qquad
		1\le j\le
		\left\lfloor\frac{n+1}{2}\right\rfloor.
	\end{equation}
	Equivalently,
	\begin{equation}\label{eq:path-factor}
		F_n(z)=
		\prod_{j=1}^{\lfloor(n+1)/2\rfloor}
		\left(
		1+4z\cos^2\frac{j\pi}{n+2}
		\right).
	\end{equation}
\end{proposition}

We shall use two consequences of Proposition~\ref{prop:path-roots}.
First, $F_n(z)\neq0$ whenever $\Ree z\ge0$. Second,
\eqref{eq:path-roots} determines the ordering of the zeros of
consecutive path polynomials. The particular orderings needed below
will be recorded in Section~\ref{sec:phasegain}.

We next express the independence polynomial of a spider in terms of
path polynomials. Let $S=S(\ell_1,\ldots,\ell_d)$ have center $c$.
If an independent set does not contain $c$, then the $j$th leg
contributes $F_{\ell_j}(z)$. If it contains $c$, then the vertex
adjacent to $c$ on each leg cannot be chosen, and the remaining part
of the $j$th leg contributes $F_{\ell_j-1}(z)$. Since the chosen
center contributes a factor $z$, we have
\begin{equation}\label{eq:spider-poly}
	I(S,z)=A(z)+B(z),
\end{equation}
where
\begin{equation}\label{eq:AB}
	A(z):=\prod_{j=1}^d F_{\ell_j}(z),
	\qquad
	B(z):=z\prod_{j=1}^d F_{\ell_j-1}(z).
\end{equation}

The two products in \eqref{eq:AB} naturally lead to ratios of
consecutive path polynomials. For $n\ge1$, define
\begin{equation}\label{eq:H-def}
	H_n(z):=\frac{F_n(z)}{F_{n-1}(z)},
	\qquad H_0(z):=1.
\end{equation}
%By Proposition~\ref{prop:path-roots}, $F_{n-1}$ has no zero in the
%closed right half-plane, so $H_n$ is well defined there.
For $n\ge2$, dividing \eqref{eq:path-rec} by $F_{n-1}(z)$ gives the
recurrence below. It also holds for $n=1$, since
$H_1(z)=F_1(z)/F_0(z)=1+z$ and $H_0(z)=1$. Thus
\begin{equation}\label{eq:H-rec}
	H_n(z)=1+\frac{z}{H_{n-1}(z)}
	\qquad(n\ge1).
\end{equation}

We first record a simple property of these ratios on the positive
imaginary axis.

\begin{lemma}\label{lem:first-quadrant}
	For every $n\ge1$ and $y>0$, the number $H_n(iy)$ lies in the open
	first quadrant and satisfies $|H_n(iy)|>1$. Consequently, there are
	unique numbers $R_n(y)>1$ and $\beta_n(y)\in(0,\pi/2)$ such that
	\[
	H_n(iy)=R_n(y)e^{i\beta_n(y)}.
	\]
\end{lemma}

\begin{proof}
	For $n=1$, we have $H_1(iy)=1+iy$, so the assertion is immediate.
	Suppose that $H_{n-1}(iy)=a+ib$ with $a,b>0$. By
	\eqref{eq:H-rec},
	\[
	H_n(iy)
	=1+\frac{iy}{a+ib}
	=1+\frac{yb}{a^2+b^2}
	+i\frac{ya}{a^2+b^2}.
	\]
	Hence
	\[
	\Ree H_n(iy)
	=1+\frac{yb}{a^2+b^2}>1,
	\qquad
	\Imm H_n(iy)
	=\frac{ya}{a^2+b^2}>0.
	\]
	Thus $H_n(iy)$ lies in the open first quadrant, and
	$|H_n(iy)|>\Ree H_n(iy)>1$. The result follows by induction.
\end{proof}

The following estimate is the main technical ingredient needed in the
proof of Theorem~\ref{thm:main}. It compares the modulus and argument
of $H_n(iy)$. Since its proof is considerably longer than the other
arguments in this section, we postpone it to
Section~\ref{sec:phasegain}.

\begin{lemma}\label{lem:phase-gain}
	For every $n\ge1$ and $y\ge1$, let
	$H_n(iy)=R_n(y)e^{i\beta_n(y)}$ as in
	Lemma~\ref{lem:first-quadrant}. Then
	\begin{equation}\label{eq:phase-gain}
		\log R_n(y)>
		\frac{2\beta_n(y)}{3\pi}\log y.
	\end{equation}
\end{lemma}

\section{Stability of Spiders}
\label{sec:mainproof}

We shall use the following standard form of the argument principle;
see, for example, \cite{Fisher1990}.

\begin{theorem}\cite{Fisher1990}\label{thm:argument-principle}
	Let $D$ be a bounded domain whose boundary $\Gamma$ is a positively
	oriented piecewise smooth simple closed curve. Suppose that $f$ is
	holomorphic in a neighborhood of $\overline{D}$ and has no zero on
	$\Gamma$. Then
	\[
	\frac{1}{2\pi i}
	\int_{\Gamma}\frac{f'(z)}{f(z)}\,dz
	\]
	is the number of zeros of $f$ in $D$, counted with multiplicity.
\end{theorem}

Let $S=S(\ell_1,\ldots,\ell_d)$ be a spider, and let $A(z)$ and $B(z)$ be as
in \eqref{eq:AB}, so that $I(S,z)=A(z)+B(z)$. Consider the
one-parameter family
\begin{equation}\label{eq:homotopy}
	J_t(z):=A(z)+tB(z),
	\qquad 0\le t\le1.
\end{equation}
Thus $J_0(z)=A(z)$, whose zeros are all negative real numbers, while
$J_1(z)=I(S,z)$.

We first show that no polynomial in this family has a zero on the
imaginary axis.

\begin{lemma}\label{lem:no-imaginary}
	For every $t\in[0,1]$ and every $y\in\mathbb R$,
	$J_t(iy)\neq0$.
\end{lemma}

\begin{proof}
	Since $J_t(0)=1$, it is enough to consider $y\neq0$. The polynomial
	$J_t$ has real coefficients, so by conjugation symmetry we may assume
	that $y>0$.

	Suppose, to the contrary, that $J_t(iy)=0$. By
	Proposition~\ref{prop:path-roots}, together with $F_0=1$, none of the
	factors $F_{\ell_j-1}(iy)$ vanishes. Dividing by their product gives
	\begin{equation}\label{eq:boundary-product}
		\prod_{j=1}^d H_{\ell_j}(iy)=-ity.
	\end{equation}
	The left-hand side is nonzero, so necessarily $t>0$.

	Write $H_{\ell_j}(iy)=R_{\ell_j}(y)e^{i\beta_{\ell_j}(y)}$ as in
	Lemma~\ref{lem:first-quadrant}. Taking absolute values in
	\eqref{eq:boundary-product} gives
	$\prod_{j=1}^d R_{\ell_j}(y)=ty\le y$. On the other hand, taking
	arguments gives $
	\sum_{j=1}^d\beta_{\ell_j}(y)
	=\frac{3\pi}{2}+2k\pi$
	for some $k\ge0$. In particular,
	$\sum_{j=1}^d\beta_{\ell_j}(y)\ge3\pi/2$.

	If $0<y<1$, Lemma~\ref{lem:first-quadrant} gives
	$R_{\ell_j}(y)>1$ for every $j$. Hence
	$\prod_{j=1}^dR_{\ell_j}(y)>1>y$, contradicting
	$\prod_{j=1}^dR_{\ell_j}(y)\le y$.
	It remains to consider $y\ge1$. By Lemma~\ref{lem:phase-gain},
	\begin{align*}
		\log\prod_{j=1}^dR_{\ell_j}(y)
		=\sum_{j=1}^d\log R_{\ell_j}(y)
		>\frac{2\log y}{3\pi}
		\sum_{j=1}^d\beta_{\ell_j}(y)
		\ge\log y.
	\end{align*}
	Therefore $\prod_{j=1}^dR_{\ell_j}(y)>y$, again a contradiction.
\end{proof}
The next lemma provides a suitable starting point for the
zero-counting argument used in the proof of the main theorem.

\begin{lemma}\label{lem:starting-point}
	There exists $t_0\in[0,1)$ such that all zeros of $J_{t_0}(z)$ lie
	in the open left half-plane and the degree of $J_t(z)$ is constant
	for $t\in[t_0,1]$.
\end{lemma}

\begin{proof}
	Let
	\[
	o:=\#\{j:\ell_j\text{ is odd}\}.
	\]
	If $\ell_j$ is even, then $F_{\ell_j}$ and $F_{\ell_j-1}$ have the
	same degree. If $\ell_j$ is odd, then the degree of
	$F_{\ell_j-1}$ is one less than that of $F_{\ell_j}$. Since $B(z)$
	has one additional factor $z$, \eqref{eq:path-degree} gives
	\begin{equation}\label{eq:degree-difference}
		\deg B(z)-\deg A(z)=1-o.
	\end{equation}

	Suppose first that $o\ge1$. If $o\ge2$, then
	$\deg B(z)<\deg A(z)$, and hence $J_t(z)=A(z)+tB(z)$ has the same degree as
	$A(z)$ for every $t\in[0,1]$. If $o=1$, then
	$\deg A(z)=\deg B(z)$. The leading coefficients of both $A(z)$ and $B(z)$
	are positive by \eqref{eq:path-coeff}, so their highest-degree
	terms cannot cancel in $A(z)+tB(z)$ for $t\ge0$. Thus $J_t(z)$ again has
	constant degree on $[0,1]$. In either case, $J_0(z)=A(z)$ is a product of path polynomials and
	therefore has only negative real zeros by
	Proposition~\ref{prop:path-roots}. Hence we may take $t_0=0$.

	It remains to consider $o=0$. Then every leg has even length.
	Write $\ell_j=2m_j$ and let
	\[
	M:=\sum_{j=1}^d m_j,
	\qquad
	C:=\prod_{j=1}^d(m_j+1).
	\]
	By \eqref{eq:path-degree},
	\begin{equation}\label{eq:all-even-leading}
		A(z)=Cz^M+O(z^{M-1}),
		\qquad
		B(z)=z^{M+1}+O(z^M).
	\end{equation}
	Hence $\deg J_0(z)=M$, whereas $\deg J_t(z)=M+1$ for every $t>0$.
	We shall show that all zeros of $J_t(z)$ lie in the open left
	half-plane when $t>0$ is sufficiently small.

	Let $r_1,\ldots,r_s$ be the distinct zeros of $A(z)$. By
	Proposition~\ref{prop:path-roots}, they are all negative real
	numbers. Choose pairwise disjoint closed disks
	$D_1,\ldots,D_s$, centered at these zeros and contained in the
	open left half-plane, with no zero of $A(z)$ on their boundaries.
	Let $\Gamma=\bigcup_{q=1}^s\partial D_q$. Since $\Gamma$ is
	compact,
	\[
	m_A:=\min_{z\in\Gamma}|A(z)|>0,
	\qquad
	M_B:=\max_{z\in\Gamma}|B(z)|<\infty.
	\]
	Choose $\varepsilon_0>0$ such that
	$\varepsilon_0M_B<m_A$. Then
	$|tB(z)|<|A(z)|$ on $\Gamma$ whenever
	$0<t\le\varepsilon_0$.

	By Rouch\'e's theorem~\cite{Fisher1990}, for
	$0<t\le\varepsilon_0$, the polynomial $J_t(z)=A(z)+tB(z)$ has exactly $M$ zeros, counted with
	multiplicity, in these disks. All of them therefore lie in the open
	left half-plane. Since $\deg J_t(z)=M+1$, there is exactly one remaining
	zero, say $\zeta_t$. Since $J_t(z)$ has real coefficients, this remaining zero must be real (nonreal zeros come in conjugate pairs).

	Let $b_M=[z^M]B$. By Vieta's formula,
	\[
	\sum_{J_t(\zeta)=0}\zeta
	=-\frac{C+tb_M}{t}
	=-\frac{C}{t}-b_M.
	\]
	The $M$ zeros in the disks remain bounded as $t\to0^+$. Hence
	\[
	\zeta_t=-\frac{C}{t}+O(1).
	\]
	Since $C>0$, we have $\zeta_t\to-\infty$ as $t\to0^+$.
	Hence, for all sufficiently small $t>0$, the remaining zero
	$\zeta_t$ also lies in the open left half-plane. Therefore all zeros
	of $J_t(z)$ lie in the open left half-plane for all sufficiently small
	$t>0$.

	Choose such a $t_0>0$. Since $\deg J_t(z)=M+1$ for every $t>0$, the
	degree of $J_t(z)$ is constant on $[t_0,1]$.
\end{proof}

We now prove the main result.
\begin{proof}[Proof of Theorem~\ref{thm:main}]
	Let $t_0$ be given by Lemma~\ref{lem:starting-point}, and let $N$
	be the common degree of $J_t(z)$ for $t\in[t_0,1]$. Write
	\[
	J_t(z)=a_N(t)z^N+\cdots+a_0(t).
	\]
	Since the coefficient functions are continuous on the compact
	interval $[t_0,1]$ and $a_N(t)\neq0$ there, there exist constants
	$c,M>0$ such that
	\[
	|a_N(t)|\ge c,
	\qquad
	|a_k(t)|\le M\quad(0\le k<N)
	\]
	for every $t\in[t_0,1]$.

	For $|z|\ge1$,
	\[
	\left|\sum_{k=0}^{N-1}a_k(t)z^k\right|
	\le NM|z|^{N-1}.
	\]
	Hence, if $|z|>\max\{1,NM/c\}$, then
	\[
	|a_N(t)z^N|
	>
	\left|\sum_{k=0}^{N-1}a_k(t)z^k\right|,
	\]
	so $J_t(z)\neq0$. Thus there is $R_0>0$ such that every zero of
	every $J_t(z)$, $t\in[t_0,1]$, lies in $|z|\le R_0$.

	Choose $R>R_0$, and let $\Gamma_R$ be the positively oriented
	boundary of
	\[
	D_R^+:=\{z\in\mathbb C:\Ree z>0,\ |z|<R\}.
	\]
	By the choice of $R$ and Lemma~\ref{lem:no-imaginary},
	$J_t(z)$ has no zero on $\Gamma_R$ for any $t\in[t_0,1]$.
	By Theorem~\ref{thm:argument-principle},
	\[
	N_+(t):=
	\frac{1}{2\pi i}
	\int_{\Gamma_R}
	\frac{J_t'(z)}{J_t(z)}\,dz
	\]
	is the number of zeros of $J_t(z)$ in the open right half-plane,
	counted with multiplicity.

	Since $J_t(z)\neq0$ on $\Gamma_R$ for all $t\in[t_0,1]$, the
	integrand depends continuously on $(t,z)$ on the compact set
	$[t_0,1]\times\Gamma_R$. Hence $N_+(t)$ is continuous in $t$.
	Since it is integer-valued, it is constant on $[t_0,1]$.

	By Lemma~\ref{lem:starting-point}, $N_+(t_0)=0$, and hence
	$N_+(1)=0$. Lemma~\ref{lem:no-imaginary} also excludes zeros of
	$J_1(z)$ on the imaginary axis. Since $J_1(z)=I(S,z)$, every
	independence root of $S$ has strictly negative real part.
\end{proof}

\section{Proof of Lemma \ref{lem:phase-gain}}
\label{sec:phasegain}

We now prove Lemma~\ref{lem:phase-gain}. We begin with two auxiliary
facts. The first gives a representation of $H_n(iy)$ that will be used
in both parity cases.

Fix $y>0$, and let $\lambda$ and $\mu$ be the roots of
\[
w^2-w-iy=0.
\]
Thus
\begin{equation}\label{eq:fixed-vieta}
	\lambda+\mu=1,
	\qquad
	\lambda\mu=-iy.
\end{equation}
Since $y>0$, the two roots are distinct. They also cannot have the
same modulus. Indeed, if $|\lambda|=|\mu|$, then
$|\lambda|=|1-\lambda|$, and hence $\Ree\lambda=1/2$. It follows that
$\mu=1-\lambda=\overline{\lambda}$, so $\lambda\mu$ would be real,
contrary to \eqref{eq:fixed-vieta}. We therefore label the roots so
that $|\lambda|>|\mu|$, and set
\[
\kappa:=\frac{\mu}{\lambda},
\qquad
r:=|\kappa|\in(0,1),
\qquad
L:=-\log r>0.
\]

\begin{lemma}\label{lem:ratio-representation}
	For every $n\ge0$,
	\begin{equation}\label{eq:H-ratio}
		H_n(iy)
		=
		\lambda\,
		\frac{1-\kappa^{n+2}}{1-\kappa^{n+1}}.
	\end{equation}
	Moreover,
	\begin{equation}\label{eq:y-r}
		y=\frac{r(1+r^2)}{(1-r^2)^2},
		\qquad
		|\lambda|=\sqrt{\frac{y}{r}}.
	\end{equation}
	If $a:=(1-r)/(1+r)$, then
	\begin{equation}\label{eq:y-a}
		y=\frac{1-a^4}{8a^2},
		\qquad
		a^2=
		\frac{1}{\sqrt{16y^2+1}+4y},
		\qquad
		L=2\operatorname{artanh}a>2a.
	\end{equation}
\end{lemma}

\begin{proof}
	Since $\lambda$ and $\mu$ satisfy $w^2=w+iy$, for $n\ge2$ we have
	\[
	\lambda^{n+2}=\lambda^{n+1}+iy\,\lambda^n,
	\qquad
	\mu^{n+2}=\mu^{n+1}+iy\,\mu^n.
	\]
	Hence the sequence
	\[
	G_n:=
	\frac{\lambda^{n+2}-\mu^{n+2}}{\lambda-\mu}
	\]
	satisfies
	\[
	G_n=G_{n-1}+iy\,G_{n-2}
	\qquad(n\ge2),
	\]
	which is the same recurrence as $F_n(iy)$. Moreover,
	\[
	G_0=\lambda+\mu=1,
	\qquad
	G_1
	=\lambda^2+\lambda\mu+\mu^2
	=(\lambda+\mu)^2-\lambda\mu
	=1+iy.
	\]
	Thus $G_0=F_0(iy)$ and $G_1=F_1(iy)$, and the recurrence gives
	\[
	F_n(iy)
	=
	\frac{\lambda^{n+2}-\mu^{n+2}}{\lambda-\mu}
	\qquad(n\ge0).
	\]
	For $n\ge1$, dividing consecutive terms yields
	\[
	H_n(iy)
	=
	\lambda
	\frac{1-(\mu/\lambda)^{n+2}}
	{1-(\mu/\lambda)^{n+1}},
	\]
	which is \eqref{eq:H-ratio}. For $n=0$, the right-hand side of
	\eqref{eq:H-ratio} equals
	\[
	\lambda\frac{1-\kappa^2}{1-\kappa}
	=\lambda(1+\kappa)
	=\lambda+\mu
	=1
	=H_0(iy),
	\]
	so the formula also holds for $n=0$.

	We next relate $r$ to $y$. Since $\mu=\kappa\lambda$ and
	$\lambda+\mu=1$, we have $\lambda=(1+\kappa)^{-1}$. Together with
	$\lambda\mu=-iy$, this gives
	\[
	\frac{\kappa}{(1+\kappa)^2}=-iy.
	\]
	Write $\kappa=re^{i\theta}$. Since the right-hand side is purely
	imaginary,
	\[
	\Ree\!\bigl(\kappa(1+\overline{\kappa})^2\bigr)=0.
	\]
	Expanding the real part gives
	\[
	r(1+r^2)\cos\theta+2r^2=0,
	\]
	and hence
	\[
	\cos\theta=-\frac{2r}{1+r^2}.
	\]
	It follows that
	\[
	|1+\kappa|^2
	=1+r^2+2r\cos\theta
	=\frac{(1-r^2)^2}{1+r^2}.
	\]
	Therefore
	\[
	|\lambda|^2
	=\frac{1+r^2}{(1-r^2)^2}.
	\]
	On the other hand, taking absolute values in
	$\lambda\mu=-iy$ gives $y=r|\lambda|^2$. These two relations prove
	\eqref{eq:y-r}.

	Finally, let $a=(1-r)/(1+r)$. Then $r=(1-a)/(1+a)$, and
	substitution into \eqref{eq:y-r} gives
	\[
	y=\frac{1-a^4}{8a^2}.
	\]
	Equivalently, $a^4+8ya^2-1=0$, so
	\[
	a^2
	=\sqrt{16y^2+1}-4y
	=\frac{1}{\sqrt{16y^2+1}+4y}.
	\]
	Also,
	\[
	L=-\log r
	=\log\frac{1+a}{1-a}
	=2\operatorname{artanh}a.
	\]
	Since
	\[
	\operatorname{artanh}a
	=\int_0^a\frac{dt}{1-t^2}>a,
	\]
	we obtain $L>2a$.
\end{proof}

The second auxiliary fact is an elementary estimate that will be used
in the even case.

\begin{lemma}\label{lem:tanh-estimate}
	For every $s>0$,
	\[
	\tanh\!\sqrt{\frac{2\pi}{3s}}>e^{-s/6}.
	\]
\end{lemma}

\begin{proof}
	For $x>0$, putting $u=e^{-2x}$ and using the power series for
	$\log(1+u)$ and $-\log(1-u)$, we obtain
	\[
	\log\coth x
	=
	\log\frac{1+u}{1-u}
	=
	2\sum_{k=0}^{\infty}
	\frac{e^{-2(2k+1)x}}{2k+1}.
	\]
	For every positive integer $q$, the function
	$x^2e^{-2qx}$ attains its maximum at $x=1/q$, with maximum value
	$e^{-2}/q^2$. Hence
	\[
	x^2\log\coth x
	\le
	\frac{2}{e^2}
	\sum_{k=0}^{\infty}\frac{1}{(2k+1)^3}.
	\]
	Moreover,
	\[
	\sum_{k=0}^{\infty}\frac{1}{(2k+1)^3}
	<
	1+\frac1{27}+\int_3^\infty t^{-3}\,dt
	=
	\frac{59}{54}
	<
	\frac{10}{9}.
	\]
	Therefore, using $e^2>7$ and $\pi>3$,
	\[
	x^2\log\coth x
	<
	\frac{20}{9e^2}
	<
	\frac{20}{63}
	<
	\frac13
	<
	\frac{\pi}{9}.
	\]
	
	Now take $x=\sqrt{\frac{2\pi}{3s}}.$
	Then $\pi/(9x^2)=s/6$, and so $
	\log\coth x<s/6$.
	Exponentiating and taking reciprocals gives 
	$\tanh x>e^{-s/6},$
	which proves the result.
\end{proof}

We are now ready to prove Lemma~\ref{lem:phase-gain}.

\begin{proof}[Proof of Lemma~\ref{lem:phase-gain}]
	If $y=1$, then $R_n(1)>1$ by
	Lemma~\ref{lem:first-quadrant}, while the right-hand side of
	\eqref{eq:phase-gain} is zero. We may therefore assume that
	$y>1$.

	We first record a consequence of \eqref{eq:y-a} that will be used
	in both parity cases. Since $y>1$,
	$\sqrt{16y^2+1}<5y$, and hence
	\[
	a^2
	=\frac{1}{\sqrt{16y^2+1}+4y}
	>\frac{1}{9y}.
	\]
	Together with $L>2a$, this gives
	\begin{equation}\label{eq:L-lower}
		L>\frac{2}{3\sqrt y}.
	\end{equation}

	\medskip
	\noindent\textbf{Odd indices.}
	Let $n=2m-1$. We prove the stronger estimate
	\begin{equation}\label{eq:odd-gain}
		R_{2m-1}(y)^3>y.
	\end{equation}
	Once this is proved, \eqref{eq:phase-gain} follows from
	$\beta_{2m-1}(y)<\pi/2$.

	For $m=1$, we have $H_1(iy)=1+iy$, and hence
	\[
	R_1(y)^3=(1+y^2)^{3/2}>y.
	\]
	We may therefore assume that $m\ge2$.

	By Proposition~\ref{prop:path-roots}, write the zeros of
	$F_{2m-2}$ as $-\rho_1,\ldots,-\rho_{m-1}$ and those of
	$F_{2m-3}$ as $-\sigma_1,\ldots,-\sigma_{m-1}$, where
	\[
	\rho_j:=
	\frac{1}{4\cos^2\!\left(\frac{j\pi}{2m}\right)},
	\qquad
	\sigma_j:=
	\frac{1}{4\cos^2\!\left(\frac{j\pi}{2m-1}\right)}.
	\]
	Since $1/(4\cos^2x)$ is strictly increasing on $(0,\pi/2)$ and
	\[
	\frac{j}{2m}
	<
	\frac{j}{2m-1}
	<
	\frac{j+1}{2m}
	\qquad(1\le j\le m-2),
	\]
	while
	\[
	\frac{m-1}{2m}<\frac{m-1}{2m-1},
	\]
	we obtain
	\begin{equation}\label{eq:rho-sigma}
		0<\rho_1<\sigma_1<\rho_2<\sigma_2<\cdots
		<\rho_{m-1}<\sigma_{m-1}.
	\end{equation}

	Both $F_{2m-3}$ and $F_{2m-2}$ have degree $m-1$, with leading
	coefficients $1$ and $m$, respectively. Since the zeros of
	$F_{2m-2}$ are simple, we may write
	\begin{equation}\label{eq:partial-fraction}
		\frac{F_{2m-3}(z)}{F_{2m-2}(z)}
		=
		\frac1m+
		\sum_{j=1}^{m-1}\frac{c_j}{z+\rho_j},
	\end{equation}
	where
	\[
	c_j
	=
	\frac{F_{2m-3}(-\rho_j)}
	{F'_{2m-2}(-\rho_j)}.
	\]
	The ordering \eqref{eq:rho-sigma} implies that $c_j>0$. Indeed,
	up to their positive leading coefficients,
	\[
	F_{2m-3}(-\rho_j)
	=\prod_{k=1}^{m-1}(\sigma_k-\rho_j),
	\]
	while
	\[
	F'_{2m-2}(-\rho_j)
	=
	m\prod_{\substack{1\le k\le m-1\\k\ne j}}
	(\rho_k-\rho_j).
	\]
	In each product exactly $j-1$ factors are negative. Thus the
	numerator and denominator in the expression for $c_j$ have the
	same sign, and hence $c_j>0$.

	By the path recurrence,
	\[
	H_{2m-1}(z)
	=
	1+
	z\frac{F_{2m-3}(z)}{F_{2m-2}(z)}.
	\]
	Using \eqref{eq:partial-fraction} at $z=iy$, we obtain
	\[
	H_{2m-1}(iy)
	=
	1+\frac{iy}{m}
	+\sum_{j=1}^{m-1}
	c_j\frac{iy}{\rho_j+iy}.
	\]
	For each $j$,
	\[
	\frac{iy}{\rho_j+iy}
	=
	\frac{y^2}{\rho_j^2+y^2}
	+i\frac{y\rho_j}{\rho_j^2+y^2},
	\]
	whose real and imaginary parts are both positive. Since $c_j>0$,
	\begin{equation}\label{eq:odd-real-imag}
		\Ree H_{2m-1}(iy)>1,
		\qquad
		\Imm H_{2m-1}(iy)>\frac{y}{m}.
	\end{equation}
	It follows that
	\[
	R_{2m-1}(y)^2
	>
	1+\frac{y^2}{m^2}.
	\]

	If
	\[
	m\le
	\frac{y}{\sqrt{y^{2/3}-1}},
	\]
	then
	\[
	R_{2m-1}(y)^2
	>
	1+\frac{y^2}{m^2}
	\ge
	y^{2/3},
	\]
	and therefore \eqref{eq:odd-gain} holds.

	It remains to consider
	\[
	m>
	\frac{y}{\sqrt{y^{2/3}-1}}.
	\]
	By Lemma~\ref{lem:ratio-representation},
	\[
	R_{2m-1}(y)
	=
	|\lambda|
	\frac{|1-\kappa^{2m+1}|}
	{|1-\kappa^{2m}|}.
	\]
	Using $|1-z|\ge1-|z|$ and $|1-z|\le1+|z|$, we obtain
	\[
	R_{2m-1}(y)
	\ge
	|\lambda|
	\frac{1-r^{2m+1}}{1+r^{2m}}
	>
	|\lambda|
	\frac{1-r^{2m}}{1+r^{2m}}.
	\]
	Since $|\lambda|=\sqrt{y/r}$ and $r=e^{-L}$,
	\begin{equation}\label{eq:odd-tanh}
		\frac{R_{2m-1}(y)}{y^{1/3}}
		>
		A\tanh(mL),
		\qquad
		A:=y^{1/6}r^{-1/2}.
	\end{equation}

	We first bound $A$. By \eqref{eq:y-r},
	\[
	A^6
	=
	\frac{1+r^2}{r^2(1-r^2)^2}.
	\]
	Putting $u=r^2\in(0,1)$ gives
	\[
	A^6
	>
	\frac{1}{u(1-u)^2}.
	\]
	The function $u(1-u)^2$ attains its maximum $4/27$ on $[0,1]$
	at $u=1/3$. Consequently,
	\[
	A^6>\frac{27}{4}>
	\left(\frac43\right)^6,
	\]
	and hence
	\begin{equation}\label{eq:A-fourthirds}
		A>\frac43.
	\end{equation}

	We next show that $mL>1$. For $y>1$,
	\begin{equation}\label{eq:elementary-y}
		y^{2/3}-1<\frac49y.
	\end{equation}
	Indeed, after division by $y$, this is equivalent to
	\[
	h(y):=y^{-1/3}-y^{-1}<\frac49.
	\]
	Now
	\[
	h'(y)
	=
	y^{-2}\left(1-\frac13y^{2/3}\right),
	\]
	so $h$ attains its maximum on $(1,\infty)$ at
	$y=3\sqrt3$, where
	\[
	h(3\sqrt3)
	=\frac{2}{3\sqrt3}
	<\frac49.
	\]
	This proves \eqref{eq:elementary-y}.

	Using the present assumption on $m$, together with
	\eqref{eq:L-lower} and \eqref{eq:elementary-y}, we obtain
	\[
	mL
	>
	\frac{y}{\sqrt{y^{2/3}-1}}
	\frac{2}{3\sqrt y}
	=
	\frac{2\sqrt y}
	{3\sqrt{y^{2/3}-1}}
	>1.
	\]
	Therefore, by \eqref{eq:odd-tanh} and
	\eqref{eq:A-fourthirds},
	\[
	\frac{R_{2m-1}(y)}{y^{1/3}}
	>
	\frac43\tanh1
	>1,
	\]
	where the last inequality follows from
	$\tanh1=(e^2-1)/(e^2+1)>3/4$. Thus
	\eqref{eq:odd-gain} also holds in this range.

	We have proved \eqref{eq:odd-gain} for every $m\ge1$. Since
	$\beta_{2m-1}(y)<\pi/2$ and $y>1$,
	\[
	\log R_{2m-1}(y)
	>
	\frac13\log y
	>
	\frac{2\beta_{2m-1}(y)}{3\pi}\log y.
	\]
	Hence \eqref{eq:phase-gain} holds for odd $n$.

	\medskip
	\noindent\textbf{Even indices.}
	Let $n=2m$. Define
	\[
	a_j:=
	\frac{1}{4\cos^2\!\left(\frac{j\pi}{2m+2}\right)},
	\qquad
	b_j:=
	\frac{1}{4\cos^2\!\left(\frac{j\pi}{2m+1}\right)}
	\qquad(1\le j\le m).
	\]
	For $1\le j\le m-1$,
	\[
	\frac{j}{2m+2}
	<
	\frac{j}{2m+1}
	<
	\frac{j+1}{2m+2},
	\]
	and also
	\[
	\frac{m}{2m+2}
	<
	\frac{m}{2m+1}.
	\]
	Thus Proposition~\ref{prop:path-roots} and the monotonicity of
	$1/(4\cos^2x)$ on $(0,\pi/2)$ give
	\begin{equation}\label{eq:ab-interlace}
		0<a_1<b_1<a_2<b_2<\cdots<a_m<b_m.
	\end{equation}
	By \eqref{eq:path-factor},
	\begin{equation}\label{eq:H-even-product}
		H_{2m}(z)
		=
		\prod_{j=1}^m
		\frac{1+z/a_j}{1+z/b_j}.
	\end{equation}
	Set
	\[
	E_m:=\bigcup_{j=1}^m[a_j,b_j].
	\]

	For $0<a<b$, direct differentiation gives
	\[
	\log\left|
	\frac{1+iy/a}{1+iy/b}
	\right|
	=
	\int_a^b
	\frac{y^2}{t(t^2+y^2)}\,dt
	\]
	and
	\[
	\arctan\frac{y}{a}
	-\arctan\frac{y}{b}
	=
	\int_a^b
	\frac{y}{t^2+y^2}\,dt.
	\]
	Summing these identities over $j$ gives
	\begin{align}
		\log R_{2m}(y)
		&=
		\int_{E_m}
		\frac{y^2}{t(t^2+y^2)}\,dt,
		\label{eq:even-log-int}\\
		\beta_{2m}(y)
		&=
		\int_{E_m}
		\frac{y}{t^2+y^2}\,dt.
		\label{eq:even-arg-int}
	\end{align}
	To justify the second identity, let $\Theta(y)$ denote the sum of
	the factor arguments in \eqref{eq:H-even-product}. Then
	$\Theta(y)$ is continuous for $y\ge0$, tends to $0$ as
	$y\to0^+$, and is an argument of $H_{2m}(iy)$. By
	Lemma~\ref{lem:first-quadrant}, $H_{2m}(iy)$ remains in the open
	first quadrant for $y>0$, so $\beta_{2m}(y)$ is also continuous
	and tends to $0$ as $y\to0^+$. Hence
	$\Theta(y)-\beta_{2m}(y)$ is a continuous function taking values
	in $2\pi\mathbb Z$ and tending to $0$ as $y\to0^+$.
	Therefore $\Theta(y)=\beta_{2m}(y)$.

	The largest point of $E_m$ is
	\[
	b_m
	=
	\frac{1}
	{4\sin^2\!\left(\frac{\pi}{4m+2}\right)}.
	\]
	Since $\sin x\ge2x/\pi$ for $0\le x\le\pi/2$,
	\begin{equation}\label{eq:bmax}
		b_m\le\frac{(2m+1)^2}{4}.
	\end{equation}
	For every $t\in E_m$,
	\[
	\frac{y^2}{t(t^2+y^2)}
	=
	\frac{y}{t}\frac{y}{t^2+y^2}
	\ge
	\frac{y}{b_m}\frac{y}{t^2+y^2}.
	\]
	Equality occurs only at $t=b_m$. Since $E_m$ has positive
	measure, integration together with
	\eqref{eq:even-log-int}, \eqref{eq:even-arg-int}, and
	\eqref{eq:bmax} gives
	\begin{equation}\label{eq:even-comparison}
		\log R_{2m}(y)
		>
		\frac{4y}{(2m+1)^2}\,
		\beta_{2m}(y).
	\end{equation}

	If
	\[
	(2m+1)^2
	\le
	\frac{6\pi y}{\log y},
	\]
	then
	\[
	\frac{4y}{(2m+1)^2}
	\ge
	\frac{2\log y}{3\pi},
	\]
	and \eqref{eq:even-comparison} immediately gives
	\eqref{eq:phase-gain}.

	It remains to consider
	\[
	(2m+1)^2
	>
	\frac{6\pi y}{\log y}.
	\]
	By Lemma~\ref{lem:ratio-representation},
	\[
	R_{2m}(y)
	=
	|\lambda|
	\frac{|1-\kappa^{2m+2}|}
	{|1-\kappa^{2m+1}|}.
	\]
	The triangle inequality gives
	\[
	|1-\kappa^{2m+2}|
	\ge1-r^{2m+2}
	>
	1-r^{2m+1},
	\qquad
	|1-\kappa^{2m+1}|
	\le1+r^{2m+1}.
	\]
	Since $|\lambda|=\sqrt{y/r}>\sqrt y$ and $r=e^{-L}$, we obtain
	\begin{equation}\label{eq:even-tanh}
		R_{2m}(y)
		>
		\sqrt y\,
		\frac{1-r^{2m+1}}{1+r^{2m+1}}
		=
		\sqrt y\,
		\tanh\!\left(\frac{(2m+1)L}{2}\right).
	\end{equation}

	The present assumption gives
	\[
	2m+1>
	\sqrt{\frac{6\pi y}{\log y}}.
	\]
	Combining this with \eqref{eq:L-lower}, we obtain
	\[
	\frac{(2m+1)L}{2}
	>
	\sqrt{\frac{2\pi}{3\log y}}.
	\]
	Since $\tanh$ is strictly increasing,
	Lemma~\ref{lem:tanh-estimate}, applied with $s=\log y$, gives
	\[
	\tanh\!\left(\frac{(2m+1)L}{2}\right)
	>
	\tanh\!\sqrt{\frac{2\pi}{3\log y}}
	>
	y^{-1/6}.
	\]
	Substituting this into \eqref{eq:even-tanh} yields
	\[
	R_{2m}(y)>y^{1/3}.
	\]
	Therefore
	\[
	\log R_{2m}(y)
	>
	\frac13\log y
	>
	\frac{2\beta_{2m}(y)}{3\pi}\log y,
	\]
	where the last inequality follows from
	$\beta_{2m}(y)<\pi/2$.

	Thus \eqref{eq:phase-gain} holds for even $n$ as well. Together
	with the odd case and the case $y=1$, this completes the proof.
\end{proof}

\section*{Declaration of competing interests}

The authors declare that they have no known competing financial interests or
personal relationships that could have appeared to influence the work
reported in this paper.

\section*{Acknowledgements}

\begin{sloppypar}
This work was supported by the National Natural Science Foundation of China (Grant No. 12671397) and by the Shanxi Key Laboratory of Digital Design and
Manufacturing (Grant No. 202204010931025).
\end{sloppypar}

\section*{Data availability}

No data were generated or analyzed in this study.

\end{document}